\documentclass[reqno,11pt,centertags]{article}
\usepackage{amsmath,amsthm,amscd,amssymb,latexsym,upref}
\date{\today}
\input epsf
\usepackage{epsfig}
\usepackage[T2A,OT1]{fontenc}
\usepackage[ot2enc]{inputenc}
\usepackage[russian,english]{babel}

\usepackage{epic,eepicemu}

\newcommand{\cJ}{{\mathcal{J}}}

\newcommand{\bbB}{{\mathbb{B}}}

\newcommand{\bbR}{{\mathbb{R}}}

\newcommand{\bbC}{{\mathbb{C}}}

\renewcommand{\Re}{\text{\rm Re}}

\allowdisplaybreaks \numberwithin{equation}{section}
\newtheorem{theorem}{Theorem}[section]

\newtheorem{proposition}[theorem]{Proposition}

\theoremstyle{definition}

\title{On Complex (non  analytic) Chebyshev Polynomials in $\bbC^2$}
\author{I. Moale\thanks{Supported by the Austrian Science Fund FWF, project number: P20413--N18} \ and
P. Yuditskii\thanks{Supported by the Austrian Science Fund FWF, project number: P22025--N18}}

\begin{document}

\maketitle

\centerline{\em Dedicated to the memory of Franz Peherstorfer}
\begin{abstract}
We consider the problem of finding a best uniform approximation to the standard monomial
on the unit ball in $\bbC^2$ by polynomials of lower degree with complex
coefficients.  We reduce the problem to a one-dimensional weighted minimization
problem on an interval. In a sense, the corresponding extremal polynomials are
uniform counterparts of the classical orthogonal Jacobi polynomials. They
can be represented by means of special conformal mappings on the so-called
comb-like domains. In these terms, the value of the minimal deviation and the
representation for a polynomial of best approximation for the original
problem are given. Furthermore, we derive asymptotics for the minimal deviation.
\end{abstract}

\section{Introduction}
We consider the standard basis in the set of (non analytic) complex polynomials in $\bbC^2$:
\begin{equation*}\label{eq0}
    \{z_1^{k_1}\bar z_1^{l_1}z_2^{k_2}\bar z_2^{l_2}\}_{k_1\ge 0,l_1\ge 0, k_2\ge 0,l_2\ge 0}, \quad (z_1,z_2)\in \bbC^2.
\end{equation*}
As usual $k_1+l_1+k_2+l_2$ is called the total degree of the given monomial.

In what follows we use the following notations: $\Pi_n$ denotes the set of polynomials with
complex coefficients of total degree less or equal $n,$ and $\|P\|$ denotes the uniform norm
of $P\in\Pi_n$ in the complex ball
\begin{equation*}\label{eqadd1}
   \|P\|=\sup_{(z_1,z_2)\in\bbB}|P(z_1,z_2)|, \ \ \ \bbB=\{(z_1,z_2) \in \bbC^2: |z_1|^2+|z_2|^2\le 1\}.
\end{equation*}

Analogously to the classical Chebyshev polynomial, we consider the best approximation on the ball $\mathbb B$
of the monomial $z_1^{k_1}\bar z_1^{l_1}z_2^{k_2}\bar z_2^{l_2}$ by polynomials of total degree
less than $n:=k_1+l_1+k_2+l_2.$ Such a polynomial
$\tilde{T}_{k_1,l_1,k_2,l_2}(z_1,z_2) = z_1^{k_1}\bar z_1^{l_1}z_2^{k_2}\bar z_2^{l_2} + \ldots,$
which we call a polynomial of least deviation from zero on $\mathbb B,$ is not unique
but the minimal deviation $L_{k_1,l_1,k_2,l_2}$ is well defined,
\begin{equation*}\label{eqadd2}
    L_{k_1,l_1,k_2,l_2}:=\inf_{P\in \Pi_{n-1}}\|z_1^{k_1}\bar z_1^{l_1}z_2^{k_2}\bar z_2^{l_2}-P(z_1,z_2)\|.
\end{equation*}

It is convenient to work with the normalized polynomial
$T_{k_1,l_1,k_2,l_2}:=\tilde T_{k_1,l_1,k_2,l_2}/\|\tilde T_{k_1,l_1,k_2,l_2}\|.$
Thus
\begin{equation*}\label{eq1}
    T_{k_1,l_1,k_2,l_2}(z_1,z_2)=\Lambda_{k_1,l_1,k_2,l_2}z_1^{k_1}\bar z_1^{l_1}z_2^{k_2}\bar z_2^{l_2}+\dots,
\end{equation*}
where $\Lambda_{k_1,l_1,k_2,l_2}=1/L_{k_1,l_1,k_2,l_2}$.

\begin{figure}
  \center{
  \includegraphics[scale=0.5]{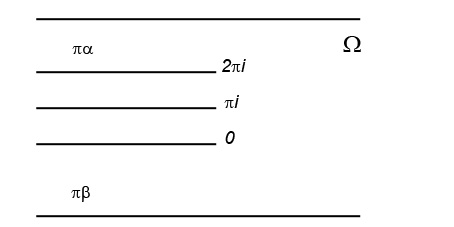}}\\
  \caption{Domain $\Omega_2(\alpha,\beta)$\label{fig1}}
\end{figure}

Concerning polynomials of least deviation from zero on the unit ball in $\bbR^2,$
several approaches are known so far, see \cite{Bra, Gea, Rei} and also \cite{MoaPeh1}.
Of foremost importance to us was the representation of the extremal polynomial
given by Bra{\ss} in \cite{Bra}. Here we essentially simplify and generalize his construction.

In approximation theory, a special role of the conformal mappings on so-called comb-like
domains is well known, see e.g. the book \cite{AKH}, the survey \cite{SYU} and the references on
original papers therein, in particular \cite{AL, Lev}.  For recent developments in this direction,
see  \cite{EYU, NPVYU}. By analogy with the MacLane-Vinberg special representation for polynomials and
entire functions \cite{McL, Vin}, in this paper to nonnegative real numbers $\alpha, \beta$ and an integer
$n \ge 0$ we associate a horizontal strip with $n+1$ horizontal cuts, see Fig. \ref{fig1}:
\begin{equation*}\label{eqadd3}
    \Omega_n(\alpha,\beta) = \{ w = u+iv : - \beta < \frac v \pi < \alpha + n \}
    \setminus \bigcup\limits_{j=0}^{n} \{ w = u+iv :  \frac v\pi= j, u \le 0 \}.
\end{equation*}
Note that the boundary of the domain contains $n+3$ infinite points:
\begin{equation*}\label{eqadd4}
\begin{split}
    \infty_0=-\infty+iv, n< \frac v \pi< n+ \alpha,&\quad  \infty_j=-\infty+iv, n-j< \frac v\pi< n-j+1, \\
    \infty_{n+1}=-\infty+iv, -\beta< \frac v \pi< 0,&\quad\infty_{n+2}=+\infty+iv, -\beta < \frac v \pi<\alpha+n.
\end{split}
\end{equation*}

Let $w: \bbC_+ \to \Omega_n(\alpha,\beta)$ be the conformal mapping of the upper half-plane onto
$\Omega_n(\alpha,\beta)$ with the following normalization
\begin{equation*}\label{eqadd5}
    w(0)=\infty_0,\quad w(1)=\infty_{n+1},\quad w(\infty)=\infty_{n+2}.
\end{equation*}
It is easy to see that it has the following asymptotics at infinity ($z\to\infty$)
\begin{equation}\label{eq2}
    w(z)=w_n(z;\alpha,\beta)=(\alpha+\beta+n)\ln z+C_n(\alpha,\beta)-\beta\pi i+ O(1/z)
\end{equation}
The real constant $C_n(\alpha,\beta)$ is uniquely defined by the domain (a kind of capacity).

Due to the evident symmetry
\begin{equation*}\label{eqsym}
    \Lambda_{k_1,l_1,k_2,l_2}=\Lambda_{l_1,k_1,k_2,l_2}=\Lambda_{k_1,l_1,l_2,k_2}
\end{equation*}
we can assume that $k_1\ge l_1$ and $k_2\ge l_2$.
Our first result is

\begin{theorem}\label{th1}
In the above introduced notations
\begin{equation*}\label{eq3}
    \ln \Lambda_{k_1,l_1,k_2,l_2} = C_{l_1+l_2}\left(\frac{k_1-l_1} 2,\frac{k_2-l_2}2\right).
\end{equation*}
\end{theorem}

Below we give the representation for a polynomial $T_{k_1,l_1,k_2,l_2}(z_1,z_2)$
of least deviation from zero. For this, we establish a connection between the conformal mapping
$w_n(z;\alpha,\beta)$ and a weighted 1-D extremal problem on $[0,1]$. In a sense, in
the following proposition we define \textit{uniform Jacobi polynomials}, compare to the
classical orthogonal ones \cite{AAR}.

\begin{proposition}\label{pr2}
Let $\xi_j=w_n^{-1}(\infty_j;\alpha,\beta)$, $1\le j\le n$. Then
\begin{equation}\label{eq5bis}
      e^{w_n(t;\alpha,\beta)}=t^\alpha(1-t)^\beta
      e^{C_{n}(\alpha,\beta)}(t-\xi_1)\dots(t-\xi_{n}),
\end{equation}
where $\tilde \cJ_n(t;\alpha,\beta):=(t-\xi_1)\dots(t-\xi_{n})=t^n+\dots$
is the polynomial of least deviation from zero on $[0,1]$ with respect to the weight $t^\alpha(1-t)^\beta$.
Moreover
\begin{equation}\label{eqnorm}
   \|\tilde \cJ_n(t;\alpha,\beta)\| =
   \sup_{0\le t\le 1}t^\alpha(1-t)^\beta |\tilde \cJ_n(t;\alpha,\beta)|=e^{-C_{n}(\alpha,\beta)},
\end{equation}
that is,
\begin{equation*}\label{eq4}
    e^{w_n(t;\alpha,\beta)}=t^\alpha(1-t)^\beta \cJ_n(t;\alpha,\beta),
\end{equation*}
as before $\cJ_n(t;\alpha,\beta):=\tilde \cJ_n(t;\alpha,\beta)/\|\tilde \cJ_n(t;\alpha,\beta)\|$.
\end{proposition}

We point out that $\xi_l<\xi_{l+1}$ for all $1\le l\le n-1$.

\begin{theorem}\label{th3}
Let $\alpha=\frac{k_1-l_1} 2\ge 0$ and $\beta=\frac{k_2-l_2} 2\ge 0.$ Let us factorize
$\tilde \cJ_{l_1+l_2}(t;\alpha,\beta) = \tilde \cJ_{l_1}^{(1)}(t)\tilde \cJ_{l_2}^{(2)}(t)$
in polynomials of degrees $l_1$ and $l_2$ respectively in the following way
\begin{equation}\label{eq5}
\begin{split}
     \tilde \cJ_{l_1}^{(1)}(t)&=(t-\xi_1)\dots(t-\xi_{l_1})\\
     \tilde \cJ_{l_2}^{(2)}(t)&=(t-\xi_{l_1+1})\dots(t-\xi_{l_1+l_2}).
\end{split}
\end{equation}
Then
\begin{equation}\label{eq6}
    T_{k_1,l_1,k_2,l_2}(z_1,z_2)=e^{C_{l_1+l_2}(\alpha,\beta)}z_1^{k_1-l_1}z_2^{k_2-l_2}
    \tilde \cJ_{l_1}^{(1)}(|z_1|^2)(-1)^{l_2}\tilde \cJ_{l_2}^{(2)}(1-|z_2|^2)
\end{equation}
\end{theorem}

Finally we present the following asymptotic relation for the value of the minimal deviation.
\begin{theorem}\label{th4}
Assume that the following limits exist
\begin{equation*}\label{eq8}
    \varkappa_1 = \lim_{n\to\infty}\frac{k_1}{n},\lambda_1=\lim_{n\to\infty}\frac{l_1}{n},
    \varkappa_2 = \lim_{n\to\infty}\frac{k_2}{n},\lambda_2=\lim_{n\to\infty}\frac{l_2}{n},
\end{equation*}
where $n=k_1+l_1+k_2+l_2$. Then
\begin{equation*}\label{eq9}
     \lim_{n\to\infty} L_{k_1,l_1,k_2,l_2}^{\frac 2 n}
    =(\lambda_1+\lambda_2)^{\lambda_1+\lambda_2}
     (\varkappa_1+\varkappa_2)^{\varkappa_1+\varkappa_2}
     (\lambda_1+\varkappa_2)^{\lambda_1+\varkappa_2}
     (\varkappa_1+\lambda_2)^{\varkappa_1+\lambda_2}.
\end{equation*}
\end{theorem}

\section{Reduction to 1-D problem}

First, we reduce our complex two-dimensional approximation problem to a weighted
approximation problem in two real variables on the standard triangle
$\Delta := \{ (t_1,t_2) \in \mathbb R^2 : t_1 \geq 0, t_2 \geq 0, t_1+ t_2 \leq 1 \}$.

For a continuous function $f$ on $\Delta,$ we define
$||f||_{\Delta} := \max\limits_{(t_1,t_2)\in \Delta} |f(t_1,t_2)|.$ By
$Y_{l_1,l_2}(t_1,t_2;\alpha,\beta) = {\rm M}_{l_1,l_2}(\alpha,\beta)t_1^{l_1} t_2^{l_2} + ...$
we denote a normalized polynomial of least deviation from zero on $\Delta$ with respect to the
weight function $t_1^{\alpha} t_2^{\beta}$.

\begin{proposition}\label{prop1}
Let $\alpha=\frac{k_1 - l_1}2$ and $\beta=\frac{k_2 - l_2}2.$ Then
\begin{equation*}
    z_1^{k_1 - l_1} z_2^{k_2 - l_2} Y_{l_1,l_2}(|z_1|^2,|z_2|^2;\alpha,\beta)
\end{equation*}
is a normalized polynomial of least deviation from zero on $\mathbb B$, that is,
\begin{equation*}\label{s21}
    \Lambda_{k_1,l_1,k_2,l_2}={\rm M}_{l_1,l_2}(\alpha,\beta).
\end{equation*}
\end{proposition}

\begin{proof}
Let us remark  that due to the symmetries of $\mathbb B,$ if $T_{k_1,l_1,k_2,l_2}(z_1,z_2)$
is a polynomial of least deviation from zero, then for any $\theta_1, \theta_2 \in [0,2 \pi],$
the polynomials
\begin{equation*}\label{p1}
     T_{k_1,l_1,k_2,l_2}(e^{i \theta_1} z_1, e^{i \theta_2} z_2)
     e^{- i (k_1 - l_1) \theta_1} e^{- i (k_2 - l_2) \theta_2}
\end{equation*}
and
\begin{equation}\label{p2}
     \frac{1}{(2 \pi)^2} \int_{0}^{2 \pi} \int_{0}^{2 \pi}
     T_{k_1,l_1,k_2,l_2}(e^{i \theta_1} z_1, e^{i \theta_2} z_2)
     e^{- i (k_1 - l_1) \theta_1} e^{- i (k_2 - l_2) \theta_2} d \theta_1 d \theta_2
\end{equation}
are also polynomials of least deviation from zero.

It is easy to see that the polynomial in \eqref{p2} is of the form
\begin{equation*}
\begin{split}
    \Lambda_{k_1,l_1,k_2,l_2} z_1^{k_1 - l_1} z_2^{k_2 - l_2}
    \left[
    |z_1|^{2 l_1} |z_2|^{2 l_2} +
    \sum_{cj_1 + j_2 \leq l_1 + l_2 - 1} a_{j_1,j_2} |z_1|^{2 j_1} |z_2|^{2 j_2}
    \right] \\
    =: \Lambda_{k_1,l_1,k_2,l_2} z_1^{k_1 - l_1} z_2^{k_2 - l_2} \tilde P_{l_1,l_2}(|z_1|^2,|z_2|^2)
    \end{split}
\end{equation*}
where $a_{j_1,j_2} \in \mathbb C$. Note that
\begin{equation*}
  \hat T_{k_1,l_1,k_2,l_2}(z_1,z_2) :=  z_1^{k_1 - l_1} z_2^{k_2 - l_2}Q_{l_1,l_2}(|z_1|^2,|z_2|^2),
\end{equation*}
where $Q_{l_1,l_2}(|z_1|^2,|z_2|^2):= \Lambda_{k_1,l_1,k_2,l_2}\Re \tilde P_{l_1,l_2}(|z_1|^2,|z_2|^2),$
is still a normalized polynomial of least deviation from zero on $\bbB$.

Since $(z_1,z_2) \in \mathbb B$ is equivalent to $(t_1,t_2) \in \Delta,$ where
$t_1 := |z_1|^2,$ $t_2 := |z_2|^2,$ we have that
\begin{equation*}
    ||\hat T_{k_1,l_1,k_2,l_2}||_{\mathbb B} =
    ||t_1^{(k_1-l_1)/2} t_2^{(k_2 - l_2)/2} Q_{l_1,l_2} ||_{\Delta},
\end{equation*}
which gives the assertion.
\end{proof}

Now we give a sufficient condition for a polynomial to be a weighted polynomial
of least deviation from zero on $\Delta$. In the next section  we show the existence of
a polynomial satisfying this condition.

As before $\cJ_{n}(t;\alpha,\beta)$ denotes  the normalized polynomial of least deviation
from zero on $[0,1]$ with respect to the weight function $t^{\alpha} (1-t)^{\beta}$.

\begin{proposition}\label{prop2}
Let $\alpha, \beta \geq 0$ and let $c_n=c_{n}(\alpha,\beta)>0$ be the leading coefficient of
$\cJ_{n}(t;\alpha,\beta)$, $\cJ_{n}(t;\alpha,\beta)= c_{n}t^{n} + ...$
If there exists a polynomial
\begin{equation}\label{s23}
    P_{l_1,l_2}(t_1,t_2) = c_{l_1+l_2}(\alpha,\beta)t_1^{l_1} t_2^{l_2} + ... \
    \text{with}\ ||t_1^{\alpha}t_2^{\beta}P_{l_1,l_2}||_{\Delta} \leq 1
\end{equation}
such  that $P_{l_1,l_2}(t,1-t) =
(-1)^{l_2} \cJ_{l_1+l_2}(t;\alpha, \beta)$ for all $t \in [0,1],$ then
\begin{equation}\label{s22}
    {\rm M}_{l_1,l_2}(\alpha,\beta)=c_{l_1+l_2}(\alpha,\beta).
\end{equation}
That is, the given $P_{l_1,l_2}$ is a normalized polynomial of least deviation from
zero on $\Delta$ with respect to the weight $t_1^{\alpha} t_2^{\beta}.$
\end{proposition}

\begin{proof}
Actually, we have to prove \eqref{s22}.

By the fact that $Y_{l_1,l_2}$ is a polynomial of least deviation from zero, from \eqref{s23}
we have immediately that $M_{l_1,l_2} \geq c_{l_1+l_2}.$

On the other hand, let us restrict $Y_{l_1,l_2}$ to the line $t_1=t, t_2=1-t:$
\begin{equation*}\label{s24}
    Q(t):=(-1)^{l_2}Y_{l_1,l_2}(t,1-t)={\rm M}_{l_1,l_2} t^{l_1 + l_2} + ...
\end{equation*}
Since $ | Q(t) t^{\alpha} (1 - t)^{\beta}| \leq 1 $ for all $t\in[0,1],$
the extremal property of  $\cJ_{l_1+l_2}$ implies $ {\rm M}_{l_1,l_2} \leq c_{l_1+l_2}.$
Thus the statement is proved.
\end{proof}

\section{Proofs of Proposition \ref{pr2}, Theorems \ref{th1} and \ref{th3}}

\begin{proof}[Proof of Proposition \ref{pr2}]
Let $\eta_k = w_n^{-1}((n - k) \pi i; \alpha, \beta),$ $0 \leq k \leq n.$
From the Schwarz-Christoffel formula, see e.g. \cite{DriTre}, we obtain the following
expression for the differential of the conformal mapping $w_n(z;\alpha,\beta):$
\begin{equation}\label{sch1}
     d w_n(z;\alpha,\beta) = C \frac{\prod\limits_{k=0}^n
     (z - \eta_k)}{z (z - 1) \prod\limits_{j=1}^n (z-\xi_j) } dz,
\end{equation}
where $C \in \mathbb C$ is a constant. Having in mind the asymptotic behavior at the infinite
boundary points of the domain $\Omega_n(\alpha,\beta)$ we get the following expansion into
partial fraction for \eqref{sch1}:
\begin{equation*}\label{sch2}
     d w_n(z;\alpha,\beta) = \left( \frac{\alpha}{z} + \frac{\beta}{z - 1} +
     \sum_{j=1}^n \frac 1 {z - \xi_j} \right) dz.
\end{equation*}
Hence
\begin{equation}\label{sch3}
     w_n(z;\alpha,\beta) = \alpha \ln z + \beta \ln (z - 1) + \sum_{j=1}^n \ln(z-\xi_j) + C_1,
\end{equation}
where $C_1 \in \mathbb C$ is a constant. Relation \eqref{eq5bis} follows now immediately from \eqref{sch3}.

From the boundary correspondence for the given conformal mapping we get that the function
$t^{\alpha} (1-t)^{\beta}\cJ_n(t;\alpha,\beta)$ alternates $n+1$ times between $\pm 1$  on $[0,1]$,
see Fig. \ref{fig2}. Thus the Chebyshev alternation theorem implies that $\cJ_n(t;\alpha,\beta)$
is indeed the polynomial of least deviation from zero with respect to the given weight with
leading coefficient $c_n(\alpha,\beta)=e^{C_n(\alpha,\beta)}.$
\begin{figure}
  \center{
  \includegraphics[scale=0.23]{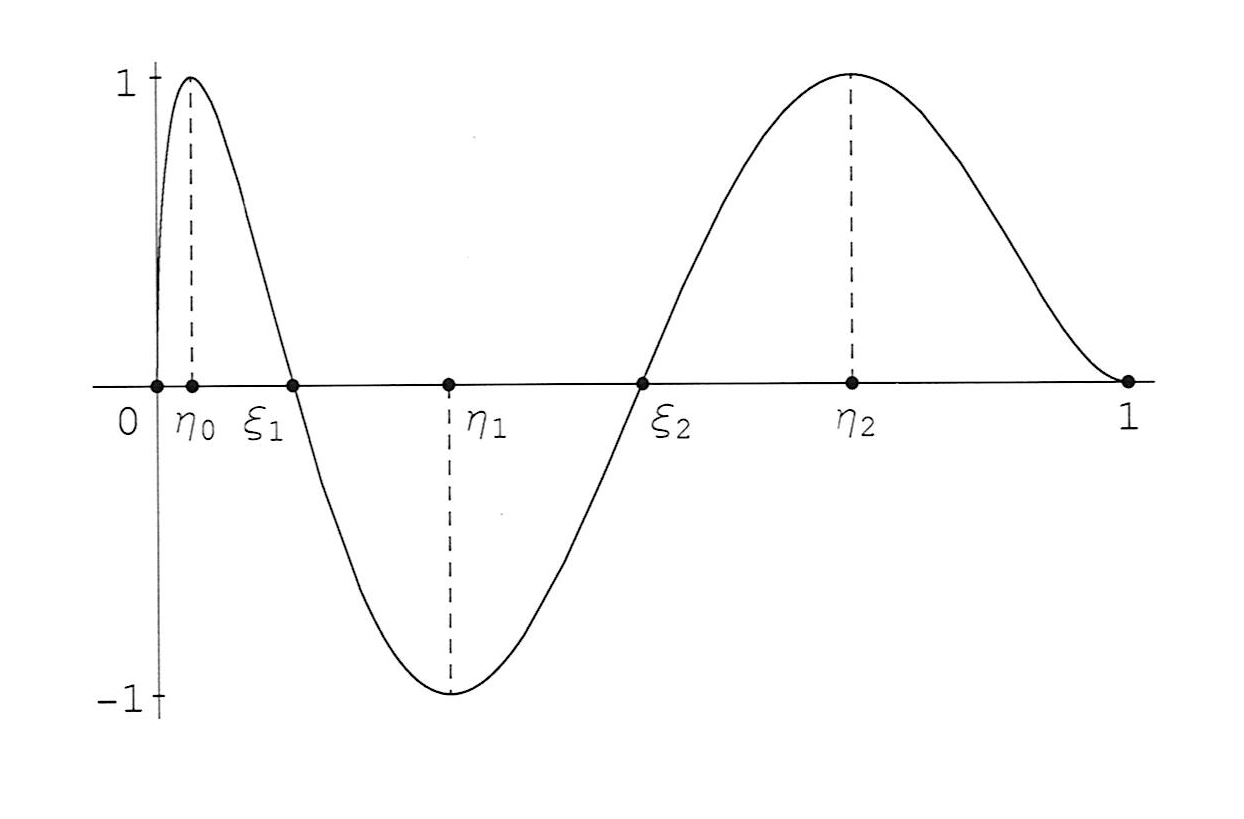}}\\
  \caption{Graph of  $t^\alpha(1-t)^\beta \cJ_2(t;\alpha,\beta)$, $\alpha=1/2, \beta=2$. \label{fig2}}
\end{figure}
\end{proof}

\begin{proof}[Proof of Theorems \ref{th1} and \ref{th3}]
By Proposition \ref{prop1}, the assertion of Theorem \ref{th3} follows if we are able to show that
\begin{equation}\label{polaux}
\begin{split}
       P_{l_1,l_2}(t_1,t_2;\alpha,\beta)
     & := e^{C_{l_1+l_2}(\alpha,\beta)} \tilde{J}^{(1)}_{l_1}(t_1)(-1)^{l_2}\tilde{J}^{(2)}_{l_2}(1-t_2) \\
     & = e^{C_{l_1+l_2}(\alpha,\beta)} t_1^{l_1} t_2^{l_2} + \ldots \\
\end{split}
\end{equation}
is a normalized polynomial of least deviation from zero on $\Delta$ with respect to the weight
$t_1^{\alpha} t_2^{\beta},$ for which we will use Proposition \ref{prop2}.

By restricting the polynomial $P_{l_1,l_2}$ to the line $t_1:=t,$ $t_2:=1-t,$ we obtain that for
all $t \in [0,1]:$
\begin{equation}\label{eqint}
P_{l_1,l_2}(t,1-t;\alpha,\beta)  = (-1)^{l_2} \cJ_{l_1+l_2}(t;\alpha,\beta).
\end{equation}
Thus it remains to show that $|| t_1^{\alpha} t_2^{\beta} P_{l_1,l_2} ||_{\Delta} \leq 1$.

Let
\begin{equation*}
    F(t_1,t_2;\alpha,\beta) := t_1^{\alpha} t_2^{\beta} P_{l_1,l_2}(t_1,t_2;\alpha,\beta).
\end{equation*}
Clearly $F(t_1,t_2;\alpha,\beta)$ is a product of two univariate functions, see \eqref{polaux}.
We normalize the first factor $f_1(t_1)$ by the condition
$f_1(\eta_{l_1})= 1$. Due to the definition of $\eta_{l_1}$  we have $F(\eta_{l_1},1-\eta_{l_1};\alpha,\beta)=(-1)^{l_2}e^{w_n(\eta_{l_1};\alpha,\beta)}=1$.
Thus
\begin{equation*}
    F(t_1,t_2;\alpha,\beta) = f_1(t_1) f_2(t_2), \quad  f_2(1 - \eta_{l_1})= 1(=f_1(\eta_{l_1})).
\end{equation*}
In addition, since $\xi_{l_1}<\eta_{l_1}$ we can easily check, see \eqref{eq5} and \eqref{eqint},
that $f_1(t)$ is strictly increasing for $t \in [\eta_{l_1},1]$, in particular $f_1(t) \geq f_1(\eta_{l_1})=1$
here. Since  $1-\xi_{l_1+1}<1-\eta_{l_1}$, $f_2(t)$ is strictly increasing and $f_2(t) \geq 1$ for $t \in [1-\eta_{l_1},1]$.

We note that by \eqref{eqint} and \eqref{eqnorm}:
\begin{equation}\label{eqint2}
|F(t,1-t;\alpha,\beta)| = t^{\alpha} (1-t)^{\beta}| {J}_{l_1+l_2}(t;\alpha,\beta)| \leq 1,
\end{equation}
for all $t \in [0,1].$
In order to show the main claim
\begin{equation}\label{est}
|F(t_1,t_2;\alpha,\beta)| \leq 1
\end{equation}
for all $(t_1,t_2)\in \Delta,$ we distinguish three regions in $\Delta.$

If $t\in [0,\eta_{l_1}]$  then \eqref{eqint2} yields
$|f_1(t)| \leq 1/ |f_2(1 - t)| \leq 1,$ where the last inequality follows by the above
listed properties of $f_2$.
Similarly we obtain $|f_2(t)| \leq 1,$ if  $t \in [0,1-\eta_{l_1}]$. Thus
\begin{equation*}
    |F(t_1,t_2;\alpha,\beta)|=|f_1(t_1)f_2(t_2)|\le 1 \quad \text{for}\ t_1\in [0,\eta_{l_1}], t_2\in [0,1-\eta_{l_1}].
\end{equation*}

If $\eta_{l_1}\le t_1 \le 1-t_2\le 1$, then since $f_1$ is increasing
on $[\eta_{l_1},1],$ it follows that $|f_1(t_1) f_2(t_2)| \leq |f_1(1-t_2) f_2(t_2)| = |F(1-t_2,t_2;\alpha,\beta)|,$
hence \eqref{est} follows by \eqref{eqint2}.

If  $1-\eta_{l_1}\le t_2\le 1-t_1\le 1$, then since $f_2$ is increasing
on $[1-\eta_{l_1},1],$ it follows that $|f_1(t_1) f_2(t_2)| \leq |f_1(t_1) f_2(1-t_1)| = |F(t_1,1-t_1;\alpha,\beta)|,$
hence \eqref{est} follows again by \eqref{eqint2}.

By combining the three cases it follows that relation \eqref{est} holds for all $(t_1,t_2)\in \Delta.$

In conclusion, relations \eqref{eqint} and \eqref{est} being proved, by Proposition
\ref{prop2} it follows that $P_{l_1,l_2}(t_1,t_2;\alpha,\beta)$ is a normalized polynomial
of least deviation from zero on $\Delta$ with respect to the weight $t_1^{\alpha} t_2^{\beta},$
and hence, the polynomial given by \eqref{eq6} is a polynomial of least deviation from zero on $\mathbb B,$
which also proves Theorem \ref{th1}.
\end{proof}

\section{Leading term in asymptotics}

We need certain properties of the conformal mapping $w_{*}$ of the upper half-plane onto the domain
\begin{equation*}\label{star}
    \Omega_*  = \{ w = u + i v : - \beta \pi < v < (1 - \beta) \pi \}
    \setminus  \{ w = u + i v :  u \le 0,  0 \leq v \leq \alpha \pi \},
\end{equation*}
see Fig. \ref{fig3}. Due to the Schwarz-Christoffel formula \cite{DriTre}, it is of the form
\begin{equation*}\label{wst1}
    w_*(z;\alpha,\beta)=\int_{x_2}^z\frac{\sqrt{(z-x_1)(z-x_2)}}{z(z-1)}dz
\end{equation*}
where $x_1, x_2$, $0<x_1<x_2<1$, are the preimages of the angle-points $\pi \alpha i$ and $0$
respectively. As before three "infinite points" in the domain  correspond to $0, 1$ and $\infty$
and due to the size of corresponding strips we have  the following relations
\begin{equation}\label{ab}
    \alpha=\sqrt{x_1x_2},\quad \beta=\sqrt{(1-x_1)(1-x_2)}.
\end{equation}

\begin{figure}
  \center{
  \includegraphics[scale=0.6]{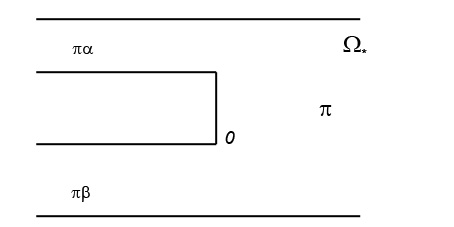}}\\
  \caption{Domain $\Omega_*(\alpha,\beta)$\label{fig3}}
\end{figure}

This is an elementary integral, so we get
\begin{equation}\label{wst10}
\begin{split}
   w_*(z;\alpha,\beta)=\sqrt{x_1x_2}\ln\frac{\left(\sqrt{x_2(1-\frac{x_1}{z})}-\sqrt{x_1(1-\frac{x_2}{z})}\right)^2}{x_2-x_1}\\
   +\sqrt{(1-x_1)(1-x_2)}\ln\frac{\left(\sqrt{(1-x_2)(1-\frac{1-x_1}{1-z})}-\sqrt{(1-x_1)(1-\frac{1-x_2}{1-z})}\right)^2}{x_1-x_2}\\
   +\ln\frac{(\sqrt{z-x_1}+\sqrt{z-x_2})^2}{x_2-x_1}.\\
 \end{split}
\end{equation}

Similarly to \eqref{eq2} we define the real constant  $C_*(\alpha,\beta)$ by the condition
\begin{equation*}\label{wsteq2}
   w_*(z;\alpha,\beta)=\ln z+C_*(\alpha,\beta)-\beta\pi i+ O(1/z), \quad z\to\infty.
\end{equation*}
By \eqref{wst10} we get
\begin{equation*}\label{wst11}
\begin{split}
   C_*(\alpha,\beta)=&\sqrt{x_1x_2}\ln\frac{(\sqrt{x_2}-\sqrt{x_1})^2}{x_2-x_1}\\
   + & \sqrt{(1-x_1)(1-x_2)}\ln\frac{\left(\sqrt{1-x_2}-\sqrt{1-x_1}\right)^2}{x_2-x_1}\\
   + & \ln\frac{4}{x_2-x_1},\\
\end{split}
\end{equation*}
which we simplify to
\begin{equation*}
\begin{split}
   C_*(\alpha,\beta)= (1-\sqrt{x_1x_2}-\sqrt{(1-x_1)(1-x_2)})\ln\frac{4}{x_2-x_1}\\
   + 2\sqrt{x_1x_2}\ln\frac{2}{\sqrt{x_2}+\sqrt{x_1}}
   + 2\sqrt{(1-x_1)(1-x_2)}\ln\frac{2}{\sqrt{1-x_2}+\sqrt{1-x_1}}.\\
\end{split}
\end{equation*}
Using \eqref{ab} we get
\begin{equation}\label{wst13}
\begin{split}
   C_*(\alpha,\beta)
   = & \frac{1-\alpha-\beta} 2\ln\frac{16}{(1-(\alpha+\beta)^2)(1-(\alpha-\beta)^2)}\\
   + & \alpha\ln\frac{4}{(1+\alpha)^2-\beta^2}+\beta\ln\frac{4}{(1+\beta)^2-\alpha^2}.
\end{split}
\end{equation}

\begin{proof}[Proof of Theorem \ref{th4}]
As the sequence of domains $\frac{2}{n}\Omega_{l_1+l_2}(\frac{k_1-l_1}{2},
\frac{k_2-l_2}{2})$ converges to $\Omega_{*}$ as $n \to \infty,$ it follows by
Carath\'{e}odory's theorem, see e.g. \cite{Ga}, that for the sequence of conformal mappings
it holds that
\begin{equation*}\label{wstadd1}
    w_*(z;\varkappa_1-\lambda_1,\varkappa_2-\lambda_2)=\lim_{n\to \infty}
    \frac 2 nw_{l_1+l_2}(z;\frac{k_1-l_1}2,\frac{k_2-l_2}2)
\end{equation*}
Therefore
\begin{equation*}\label{wstadd2}
\begin{split}
      \lim_{n\to \infty}\frac 2 n\ln \Lambda_{k_1,l_1,k_2,l_2}=
    & \lim_{n\to \infty}\frac 2 n C_{l_1+l_2}\left(\frac{k_1-l_1}2,\frac{k_2-l_2}2\right)\\
  = & C_*\left(\varkappa_1-\lambda_1,\varkappa_2-\lambda_2\right).
\end{split}
\end{equation*}

Since in this case
\begin{equation*}\label{wst15}
\begin{split}
   1- \alpha-\beta=&2(\lambda_1+\lambda_2)\\
   1+ \alpha+\beta=&2(\varkappa_1+\varkappa_2)\\
   1-\alpha+\beta=&2(\lambda_1+\varkappa_2)\\
   1+\alpha-\beta=&2(\varkappa_1+\lambda_2)\\
\end{split}
\end{equation*}
by \eqref{wst13} we get
\begin{equation*}\label{wst16}
\begin{split}
     C_*(\varkappa_1-\lambda_1,\varkappa_2-\lambda_2)=
   & -(\varkappa_1+\varkappa_2)\ln(\varkappa_1+\varkappa_2)\\
   & -(\lambda_1+\lambda_2)\ln(\lambda_1+\lambda_2)\\
   & -(\varkappa_1+\lambda_2)\ln(\varkappa_1+\lambda_2)\\
   & -(\lambda_1+\varkappa_2)\ln(\lambda_1+\varkappa_2)
\end{split}
\end{equation*}

\end{proof}

\bibliographystyle{amsplain}

\bigskip
{\em Institute for Analysis, Johannes Kepler University,

A-4040 Linz, Austria}

E-mail: Ionela.Moale@jku.at

E-mail: Petro.Yudytskiy@jku.at
\end{document}